\documentclass{article}   

\usepackage{fullpage}

\usepackage{graphicx}          
                             
\usepackage{bm}
  
\usepackage{amssymb,amsfonts,mathalfa}
\usepackage{apacite}
\usepackage{wrapfig}
\usepackage[subpreambles=true]{standalone}

\usepackage{natbib}
\usepackage{enumerate}
\usepackage{mathrsfs}
\usepackage{amsmath}
\usepackage{mathtools}

\usepackage{color}
\usepackage{floatrow}   

\usepackage{url}

\usepackage{enumitem}

\graphicspath{{./figures/}}

\allowdisplaybreaks

\usepackage{tabularx}

\usepackage{verbatim}
\usepackage{authblk}

\usepackage{subcaption}
\usepackage{dblfloatfix}      

\usepackage[table]{xcolor}
\usepackage[colorlinks,citecolor=black!80,urlcolor=black,bookmarks=false,hypertexnames=true]{hyperref} 

\newcommand{\R}{\mathbb{R}}{}
\newcommand{\N}{\mathbb{N}}
\newcommand{\Zp}{\mathbb{Z}^{+}}

\newcommand{\bS}{\mathbb{S}}

\newcommand{\cA}{\mathcal{A}}
\newcommand{\cB}{\mathcal{B}}
\newcommand{\cC}{\mathcal{C}}

\newcommand{\cH}{\mathcal{H}}

\newcommand{\cKL}{\mathcal{KL}}

\newcommand{\cG}{\mathcal{G}}
\newcommand{\M}{\langle M \rangle}

\newcommand{\wi}{\hat {\imath}}
\newcommand{\wj}{\hat {\jmath}}

\newcommand{\cF}{\mathcal{F}}

\newcommand{\cS}{\mathcal{S}}

\DeclareMathOperator{\co}{co}

\DeclareMathOperator{\Inn}{Int}
\DeclareMathOperator*{\argmin}{argmin}

\makeatletter
\newsavebox\myboxA
\newsavebox\myboxB
\newlength\mylenA

\newcommand*\pbar[1]{%
  \hbox{%
     \vbox{%
      \hrule height 0.7pt 
      \kern0.35ex
      \hbox{%
         \kern-0.0em
         \ensuremath{#1}%
         \kern-0.0em
      }%
     }%
  }%
} 

\newcommand*\xbar[2][0.75]{%
    \sbox{\myboxA}{$\m@th#2$}%
    \setbox\myboxB\null
    \ht\myboxB=\ht\myboxA%
    \dp\myboxB=\dp\myboxA%
    \wd\myboxB=#1\wd\myboxA
    \sbox\myboxB{$\m@th\pbar{\copy\myboxB}$}
    \setlength\mylenA{\the\wd\myboxA}
    \addtolength\mylenA{-\the\wd\myboxB}%
    \ifdim\wd\myboxB<\wd\myboxA%
       \rlap{\hskip 0.5\mylenA\usebox\myboxB}{\usebox\myboxA}%
    \else
        \hskip -0.5\mylenA\rlap{\usebox\myboxA}{\hskip 0.5\mylenA\usebox\myboxB}%
    \fi}
\makeatother

\newenvironment{proof}{\textbf{Proof.~}}{\hfill$\square$}

\allowdisplaybreaks

{}
\newtheorem{cor}{Corollary}{}
\newtheorem{rem}{Remark}{}

\newtheorem{conjecture}{Conjecture}

\newtheorem{prop}{Proposition}
\newtheorem{lemma}{Lemma}
\newtheorem{defn}{Definition}

\begin{document}

\title{Graph--Based Conditions for Feedback Stabilization of Switched and LPV Systems\footnote{This paper was not presented at any conference. This study was partially financed by the European Research Council (ERC) under the \emph{European Union's Horizon 2022 research and innovation program} under grant agreement No 864017 - L2C and by the ANR project HANDY 18-CE40-0010.\\ $^\ast$ Corresponding author. \\
\phantom{aa}\textit{Email addresses:} \textbf{\scriptsize matteo.dellarossa@uclouvain.be} (M. Della Rossa), \textbf{\scriptsize thiago.alves-lima@univ-lorraine.fr} (T. Alves Lima), \textbf{\scriptsize marc.jungers@univ-lorraine.fr} (M. Jungers), \textbf{\scriptsize raphael.jungers@uclouvain.be} (R. M. Jungers)}}


\author[1]{Matteo Della Rossa$^{\star,}$}
\author[2]{Thiago Alves Lima}
\author[2]{Marc Jungers}
\author[1]{Raphaël M.~Jungers}

\affil[2]{ICTEAM, UCLouvain, 4 Av. G. Lema\^itre, 1348 Louvain-la-Neuve, Belgium}

\affil[1]{Universit\'{e} de Lorraine, CNRS, CRAN, Nancy F-54000, France}  %

\maketitle

\begin{abstract}                          
This paper presents novel stabilizability conditions for switched linear systems with arbitrary and uncontrollable underlying switching signals. We distinguish and study two particular settings: i) the \emph{robust} case, in which the active mode is completely unknown and unobservable, and ii) the \emph{mode-dependent} case, in which the controller depends on the current active switching mode. The technical developments are based on  graph-theory tools, relying in particular on the path-complete Lyapunov functions framework. The main idea is to use directed and labeled graphs to encode Lyapunov inequalities to design robust and mode-dependent piecewise linear state-feedback controllers. This results in novel and flexible conditions, with the particular feature of being in the form of linear matrix inequalities (LMIs). Our technique thus provides a first controller-design strategy allowing piecewise linear feedback maps and piecewise quadratic (control) Lyapunov functions by means of semidefinite programming. Numerical examples illustrate the application of the proposed techniques, the relations between the graph order, the robustness, and the performance of the closed loop. 
\end{abstract}

\section{Introduction}
In this paper, given matrices $A_1, \dots, A_M\in \R^{n\times n}$ and $B_1, \dots, B_M\in \R^{n\times m}$, we consider the \emph{discrete-time switched control system}
\begin{equation}\label{eq:IntroSystem}
    x(k+1)=A_{\sigma(k)}x(k)+B_{\sigma(k)}u(k),
\end{equation}
where $\sigma:\N\to \M:=\{1,\dots, M\}$ is a switching signal. Switched systems are a popular model for hybrid or cyber-physical systems, with many applications in modern engineering; see e.g.,~\cite{Lib03, HesMor02,ShoWir07}. We are interested in the \emph{feedback stabilization problem} of~\eqref{eq:IntroSystem} under \emph{arbitrary switching}, i.e. we aim to design a feedback control stabilizing policy for~\eqref{eq:IntroSystem}, \emph{no matter} the underlying switching rule $\sigma:\N\to \M$. In recent years, stabilization of switched systems has been tackled from the perspective of designing a stabilizing \emph{switching sequence}, i.e. the signal $\sigma:\N\to \M$ is considered a control input for the ``autonomous'' system~\eqref{eq:IntroSystem} with $u(k) \equiv 0$. In this context, see~\cite{GerCol06, HuMaLIn08, LinAnt09,FiaJun14, JunMas17} and references therein, among many other results. The problem where \emph{both} the switching signal and the input $u(k)$ are available for design has also been tackled~\citep{FIACCHINI2017181}.

Instead, in this manuscript, we study the design of feedback maps $u(k)=\Phi(x(k)) $\ for which the switched closed-loop system~\eqref{eq:IntroSystem} is asymptotically stable, without any further assumption on the external switching policy. The underlying switching rule $\sigma:\N\to \M$  cannot be designed/modifiable by the user, and it can thus be seen as an external disturbance. 
Two cases can be highlighted, supposing, respectively, that the value of the switching signal is: 
\begin{itemize}[leftmargin=*]
    \item Un-observable by the designer/controller, in this case we aim to construct a \emph{robust feedback controller} $\Phi:\R^n\to\R^m$,
    \item Observable at the current instant of time, and in this case the goal is to design \emph{mode-dependent feedback controllers} $\Phi_1,\dots, \Phi_M:\R^n\to \R^m$.
\end{itemize}
\textcolor{black}{The stabilization problem has been intensively studied in a close and related context, that is, in the framework of polytopic \emph{linear parameter-varying} (LPV) systems, see \cite{BlaMia03} and references therein for a formal introduction.
While the class of LPV systems provides a more general model than~\eqref{eq:IntroSystem}, the corresponding stabilization problems are strongly related, and in some remarkable cases they are proved to be equivalent, see~\cite{BlaMiaSav07}.}
Several results have been established in the LPV framework, both for the robust and mode-dependent (in this literature, also called gain-scheduling) cases.
First, it has been proved that a quadratic Lyapunov function approach, while leading to numerically-appealing conditions, is conservative, see~\cite{DaafBer01,Bla95}. In order to have less conservative conditions in studying stability and stabilizability of~\eqref{eq:IntroSystem}, approaches based on \emph{piecewise-defined} (control) Lyapunov functions (and, thus piecewise continuous feedback maps) have been proposed: maxima and minima of quadratic functions in~\cite{GoeTeel06,GoeHu06}, polyhedral functions in~\cite{BlaMia03,BlaMiaSav07,BlaMia2008}. For an equivalence result between these two approaches, see~\cite{HuBla10}. 
While these methods provide exhaustive characterizations of the robust and mode-dependent stabilizability properties~\citep{BlaMiaSav07,HuBla10}, from a numerical point of view they are affected by two main limitations; first of all, the number of quadratics composing the candidate piecewise Lyapunov function (or, equivalently, the number of vertices of the candidate polyhedral Lyapunov function) to reach necessary conditions is theoretically unbounded. Secondly, the arising conditions are in the form of bilinear matrix inequalities (BMI), which are known to be NP-hard in general~\citep{TokOzb95}. 
\textcolor{black}{Other works have tackled related stabilization problems, with results leading to linear matrix conditions, for example considering \emph{static output feedback} with \emph{polynomial Lyapunov functions} (see~\cite{CheGar05,Che13} and references therein), or designing \emph{dynamic output} feedbacks with quadratic Lyapunov functions, such as~\cite{BlaMia09}. Herein, we focus on piecewise linear controllers.} 


In a slightly different setting, in studying \emph{stability} of~\eqref{eq:IntroSystem} (i.e., considering $B_1=\dots=B_M=0$),  
 the concept of \emph{path-complete Lyapunov functions} (PCLFs) has been introduced in~\cite{AhmJun:14}, in order to provide flexible conditions based on multiple Lyapunov functions. This framework involves a combinatorial component given by a directed graph that describes the set of Lyapunov inequalities to be verified and that has to be \emph{path-complete} in the sense that it captures every finite switching sequence. The complexity of this underlying graph can be increased by the user to reduce the conservatism of the arising stability conditions. For recent developments on this topic and more discussion regarding the relationship between graphs and conservatism, we refer to~\citep{DebDel22}. 
In this setting, in~\cite[Theorem III.8]{PhiAthAng}, it has been proved that any path-complete  Lyapunov function induces/can be used to construct a common Lyapunov function in the form of max-min of quadratics, thus building a bridge between the PCLF framework of~\cite{AhmJun:14} and the piecewise quadratic functions approach of~\cite{JohRan97,GoeHu06,GoeTeel06}. Summarizing, it is shown in~\cite{PhiAthAng} that the PCLF framework provides a \emph{compressed} representation of common Lyapunov functions, which allows for faster computation. In addition, the directed graph defining the PCLF provides a discrete design parameter for the user, which can be optimized in order to mitigate the numerical computational effort.

In this manuscript, we propose a novel approach to robust and mode-dependent stabilization of~\eqref{eq:IntroSystem}. Our method relies on path-complete Lyapunov functions theory, and, more specifically, we adapt and generalize this framework from stability to \emph{stabilizability} analysis.
The arising sufficient stabilization results, while depending on an underlying combinatorial structure (a path-complete graph), lead to \emph{linear matrix inequalities} (LMI) conditions, thus bypassing the numerical limitations of previous results~\citep{BlaMia2008,BlaMiaSav07,GoeHu06,GoeTeel06, HuBla10}. 
From an analytic point of view, our conditions lead to the construction of piecewise quadratic control Lyapunov functions, in the form of \emph{minimum} of quadratics, for both the robust and mode-dependent cases; moreover, the resulting feedback controllers are in a \emph{piecewise linear form}.
We thus build a connection with existing literature on piecewise quadratic (control) Lyapunov functions and piecewise defined controller, see for example~\cite{JohRan97,GoeHu06,HuBla10,LegRakRap21}.

We point out that  we recover, as a particular case of our conditions, the techniques proposed in~\cite{Lee06,LeeDull06, LeeKha09, EssLee14}; however, on top of being more general, our approach does not require the observation and memorization of past/future values of the switching signal (as in the memory/horizon-based control of~\cite{Lee06,LeeDull06,LeeKha09,EssLee14}). Indeed, the core of our proof technique is based on the above-mentioned results from path-complete Lyapunov functions~\citep{PhiAthAng}, which shows that a path-complete graph describes \emph{implicitly} a \emph{common Lyapunov function}, thus leading to a closed expression for the feedback maps (in a piecewise-linear form). This in particular allows us to tackle robust stabilization of LPV systems, for which the underlying time-varying parameter is un-observable.  Moreover, our stabilization techniques are valid for a class of graphs that goes beyond the memory structures used in~\cite{Lee06,LeeDull06,LeeKha09,EssLee14}, thus allowing for a larger class of stabilization certificates. For a discussion of the benefits of such generalization in the stability context, we refer to~\cite{PhiAthAng,DebDel22,DelRosJun22}.
Our theoretical developments are then illustrated, both in the robust and mode-dependent cases, with the help of numerical examples already introduced in the literature, thus allowing the comparison between our technique and existing results.

The rest of the manuscript is organized as follows: in Section~\ref{sec:Prel} we recall the necessary definitions and base results for switched systems and graph theory. In Section~\ref{Sec:RobFeedback} we derive our main results concerning robust stabilizability, while in Section~\ref{Sec:ModeDep} we present the stabilizability statement in the mode-dependent case. Section~\ref{sec:Conclu} closes the manuscript with some concluding remarks and possible directions for future research.

\textbf{Notation:} We denote by $\N$ the set of natural numbers including $\{0\}$, by $\Zp$ the set of natural numbers excluding $\{0\}$. Given $M\in \Zp$, we define $\M:=\{1,\dots, M\}$.  Given $n,m\in \Zp$, $\cC^0(\R^n,\R^m)$ is the set of continuous functions from $\R^n$ to $\R^m$. We denote by $\bS^{n\times n}$ the set of the $n\times n$ symmetric matrices, and by  $\bS_+^{n\times n}$  the set of $n\times n$ positive definite matrices. A function $U:\R^n\to \R$ is said to be \emph{positive definite} if $U(0)=0$ and $U(x)>0$ for all $x\neq 0$. It is said to be radially unbounded if $\lim_{\lambda \to +\infty} U(\lambda x)=+\infty$ for all $x\neq 0$.

\section{Preliminaries}\label{sec:Prel}
This section introduces the studied setting and recalls the necessary definitions and tools.
\subsection{Stabilization Notions and Characterizations}
Given $M\in \Zp$, and $f_1,\dots, f_M:\R^n\to \R^n$ such that $f_j(0)=0$ for all $j\in \M$, consider the \emph{switched system}
\begin{equation}\label{eq:ClosedLoopSwitchedSystem}
x(k+1)=f_{\sigma(k)}(x(k)),\;\;\;x(0)=x_0,\;\;k\in \N,
\end{equation}
where $\sigma:\N\to \M$ is an external switching signal and $x_0 \in \R^{n}$ is the initial condition. We denote by $\Phi_\sigma(k,x_0)$ the solution of~\eqref{eq:ClosedLoopSwitchedSystem} starting at $x_0\in \R^n$ and with respect to the signal $\sigma:\N\to \M$, evaluated at time $k\in \N$.
\begin{defn}\label{Def:UGASandUES}
System~\eqref{eq:ClosedLoopSwitchedSystem} is said to be \emph{uniformly globally asymptotically stable} (UGAS) if there exists a $\cKL$ function\footnote{A continuous function $\beta:\R_+\times \R_+\to \R_+$ is of \emph{class $\mathcal{KL}$} if $\beta(0,s)=0$  for all $s$, $\beta(\cdot,s)$ is strictly increasing and unbounded for all $s$, $\beta(r,\cdot)$ is decreasing and $\beta(r,s)\to 0$ as $s\to\infty$, for all $r$.} $\beta$ such that, for all $\sigma:\N\to \M$ and for all $x_0\in \R^n$, we have $|\Phi_\sigma(k,x_0)|\leq \beta(|x_0|,k)$. If the function $\beta$ is of the form $\beta(r,k):=C\gamma^k\,r$ for some $C>0$ and $\gamma\in [0,1)$ the system is said to be \emph{uniformly exponentially stable} (UES), and $\gamma$ is said to be the \emph{decay rate} of the system.
\end{defn}

Let us consider $M\in \Zp$ and a set $\cF=\{(A_i,B_i)\in \R^{n\times n}\times \R^{n\times m}\;\vert\;i\in \M\}$. We want to study the \emph{discrete-time control switched system} defined by
\begin{equation}\label{eq:SwitchedSystemInput}
x(k+1)=A_{\sigma(k)}x(k)+B_{\sigma(k)}u(k),\;\;\;k\in \N,
\end{equation}
where $\sigma:\N\to \M$ is an external \emph{switched signal} and $u:\N\to \R^m$ is a control input.
\begin{rem}[Related models: LPVs]
As far as the general class of switching signals is considered, i.e., with no further assumption on the feasible signals $\sigma\in\cS:=\{\theta:\N\to \M\}$, system~\eqref{eq:SwitchedSystemInput} can be equivalently represented in different frameworks.
Notably, we can consider the formalism of \emph{linear parameter-varying (LPV) systems}, i.e. the case of
\begin{equation}\label{eq:DiscTimeLPV}
x(k+1)= \cA(w(k))x(k)+\cB(w(k))u(k),
\end{equation}
where $w:\N\to \Lambda_M:=\{w\in \R^M_+\;\vert\; \sum_{i=1}^M w_i=1\}$ is a time-varying parameter taking values in $\Lambda_M$, the standard simplex of dimension $M-1$, and $\cA(w)=\sum_{i=1}^M w_iA_i$ and $\cB(w)=\sum_{i=1}^Mw_iB_i$.
System~\eqref{eq:DiscTimeLPV} provides a more general framework with respect to~\eqref{eq:SwitchedSystemInput}, since it allows the state to ``follow'' directions obtained as \emph{convex combination} of sub-systems in $\cF$.
Under some assumption on the proposed design methods and on the systems matrices, the stabilization techniques for~\eqref{eq:SwitchedSystemInput} are also effective for~\eqref{eq:DiscTimeLPV}, as discussed in what follows. \hfill $\triangle$
\end{rem}

\begin{defn}[Stabilization Notions]\label{defn:StabNot}
System~\eqref{eq:SwitchedSystemInput} is said to be
\begin{enumerate}[leftmargin=*]
\item \emph{Robust Feedback Stabilizable (RFS)} if there exists a $\Phi:\R^n \to \R^m$ such that the closed-loop system given by
\begin{equation}\label{eq:modeIndSys}
x(k+1)=A_{\sigma(k)}x(k)+B_{\sigma(k)}\Phi(x(k))
\end{equation}
is UGAS.\label{item:ModeInd(NONLInStab)}
\item \emph{Mode-Dependent Feedback Stabilizable (MDFS)} if there exist $\Phi_1,\dots, \Phi_M:\R^n \to \R^m$ such that the closed-loop system given by
\begin{equation}\label{eq:modeDepSys}
x(k+1)=A_{\sigma(k)}x(k)+B_{\sigma(k)}\Phi_{\sigma(k)}(x(k))
\end{equation}
is UGAS.\label{item:ModeDep(NONLInStab)}
\end{enumerate}
\end{defn}
In order to present the main equivalence results concerning robust and mode-dependent stabilization, we introduce a tailored notion of piecewise linear functions, inspired by~\cite[Problem 3.29]{boyd2004convex}.
\begin{defn}[Piecewise Linear Functions]\label{defn:piecwiseLinear}
A function $\Psi:\R^n\to \R^m$ is said to be \emph{piecewise linear} if:
\begin{itemize}[leftmargin=*]
\item It is \emph{homogeneous of degree $1$}, i.e., for any $\lambda \in \R_+$ and any $x\in \R^n$, it holds that $\Psi(\lambda x)=\lambda \Psi(x)$;
\item For some $c\in \Zp$, there exist $c$ closed convex cones $C_1,\dots C_c\subset \R^n$ and $c$ linear maps $\Xi_1,\dots, \Xi_c:\R^n\to\R^m$ such that $\bigcup_{j\in \langle c\rangle}C_j=\R^n$, $\Inn(C_i)\cap \Inn(C_j)=\emptyset$ for all $i\neq j$ and $\Psi(x)=\Xi_j(x)$ for all $x\in \Inn(C_j)$ for all $j\in \langle c \rangle$.
\end{itemize}
\end{defn}
    \textcolor{black}{We note that Definition~\ref{defn:piecwiseLinear} does \emph{not} imply continuity, therefore the piecewise linear feedback maps that we consider are possibly discontinuous, on the null measure set defined by the union of the boundaries of the cones $C_1, \dots C_c$ in the definition.}
In what follows, we recall a characterization result proved in~\cite{BlaMia03,BlaMiaSav07,BlaMia2008} and reviewed here for the sake of completeness.
\begin{prop}{\citep[Proposition 2]{BlaMiaSav07}}\label{prop:Charachterization}
Given $M\in \Zp$ and a set $\cF=\{(A_i,B_i)\in \R^{n\times n}\times \R^{n\times m}\;\vert\;i\in \M\}$, system~\eqref{eq:SwitchedSystemInput} is:
\begin{itemize}
    \item[(A)] \emph{RFS} if and only if there exists a \emph{piecewise linear} feedback control $\Phi:\R^n \to \R^m$ such that the closed-loop~\eqref{eq:modeIndSys} is UGAS.\label{item:APropBlanchini}
    \item[(B)] \emph{MDFS} if and only if there exist \emph{piecewise linear} feedback controls $\Phi_1,\dots,\Phi_M:\R^n\to \R^m$ such that the closed-loop~\eqref{eq:modeDepSys} is UGAS.\label{item:BPropBlanchini}
\end{itemize}
\end{prop}
The proof of this proposition contained in~\cite{BlaMiaSav07} relies on the construction of \emph{control-invariant polyhedral sets} and, thus, polyhedral Lyapunov functions. Note that the key result in the proposition is that the existence of \emph{piecewise linear} feedback maps is both necessary and sufficient for the stabilizability of~\eqref{eq:SwitchedSystemInput}, in both robust and mode-dependent cases.
We also note that by homogeneity (see~\cite[Section 5.3]{BacRosier} for the details), when a stabilizing piecewise linear feedback is provided, the closed loop will be \emph{uniformly exponentially stable} (UES). 
\begin{rem}
While Proposition~\ref{prop:Charachterization} completely characterizes the (robust and mode-dependent) stabilizability problem, the underlying conditions  are unsatisfactory from a numerical point of view. Indeed, since the proof of~\cite[Proposition 2]{BlaMiaSav07} relies on the construction of polyhedral control Lyapunov functions, the following limitations can be highlighted:
\begin{itemize}[leftmargin=*]
\item The number of vertices of the polyhedral level sets of the candidate Lyapunov functions has to be fixed a priori. To reach necessary conditions, the aforementioned number of vertices could be arbitrarily large.
\item Even when the number of vertices is fixed,  the arising matrix conditions turn out to be \emph{bilinear matrix inequalities}  (BMIs) which are, in general, NP-hard to solve, see the discussion in~\cite[Section 7.3.2]{BlaMia2008}.
    \end{itemize}
   
   Another classic approach (see~\cite{DaafBer01} and references therein) relies on quadratic control Lyapunov functions and linear feedbacks, i.e. supposing  the existence of $K_1,\dots, K_M\in \R^{n\times m}$ (a unique $K\in\R^{n\times m}$ in the robust case) and a $P\in \bS^{n \times n}_+$ such that
    \[
(A_i+B_iK_i)^\top P(A_i+B_iK_i)- P\prec 0,\;\;\;\forall \;i\in \M,
    \]
    which are \emph{nonlinear} matrix inequalities.
Re-parametrizing, considering the variables $Q=P^{-1}$ and $R_i=K_iP^{-1}$, and finally pre- and post-multiplying the latter inequality by Q, we obtain
$
(QA_i^\top+R_i^\top B_i^\top)Q^{-1}(A_iQ+B_iR_i)-Q\prec 0,
$
which can be rewritten, via Schur complement, as
\[
\begin{bmatrix}
Q & \star \\ A_iQ+B_iR_i & Q
\end{bmatrix} \succ 0, \;\;\forall i\in \M,
\]
leading to LMI conditions (the robust case, with a constant $K\in \R^{n\times m}$, is similar). The drawback of this well-known design technique are the following:
\begin{itemize}[leftmargin=*]
\item The existence of \emph{quadratic} common Lyapunov function (and linear feedbacks) is only sufficient for stability/stabilizability of switched systems, as highlighted by Proposition~\ref{prop:Charachterization}.
\item The change-of-variables technique is not applicable, in general, for broader contexts, for example involving multiple-(control)-Lyapunov functions. 
\end{itemize}

\end{rem}
\textcolor{black}{
For the reasons mentioned above, in the following sections, we propose a novel feedback design technique based on graph theory, which will lead to LMI-based conditions.
}

\subsection{Preliminaries on graph theory and path-complete Lyapunov functions}

In the following, we collect some graph-theory notions which will be used in our formal statements. Moreover, we recall the basic ideas of the sufficient conditions for stability of~\eqref{eq:ClosedLoopSwitchedSystem} based on \emph{path-complete Lyapunov functions}, for an overview see~\cite{PhiAthAng,AhmJun:14,PEDJ:16}.

A (labeled and directed) \emph{graph} $\cG=(S,E)$ on $\M$ is defined by a finite set of \emph{nodes} $S$ and a subset of labeled \emph{edges} $E\subset S\times S\times \M$.
\begin{defn}[Path-Complete Graphs]\label{defn:Path-Completeness}
A graph $\cG=(S,E)$ is \emph{path-complete} for $\M$ if, for any $l\geq 1$ and any ``word'' $(j_1\,\dots\,j_l)\in \M^l$, there exists a \emph{path}  $\{(s_k,s_{k+1},j_k)\}_{1\leq k\leq l}$ such that $(s_k,s_{k+1},j_k)\in E$, for each $1\leq k\leq l$.
\end{defn}
Intuitively, a graph is path-complete if any possible switching sequence can be reconstructed by walking through the edges of the graph.
In our setting, path-complete graphs will represent and encode the \emph{structure} of our stability certificates, or in other words, the structure of the inequalities among the candidate Lyapunov functions we aim to verify, as formalized in what follows.

\begin{defn}[Path-complete Lyapunov functions]\label{defn:PCLF}
Given a family $F = \{ f_1, \dots, f_M\}\subset \cC^0(\R^n, \R^n)$, a \emph{path-complete Lyapunov function} (PCLF) for $F$ is a pair $(\cG,V)$ where $\cG = (S,E)$ is a path-complete graph, and $V = \{ V_s \mid s \in S \} \subseteq \cC^0(\R^n ,\R)$ is a set of positive definite and radially unbounded functions  such that the following inequalities are satisfied: 
\begin{equation*}
\forall \, e=(a,b,i) \in E, \: \forall x \in \mathbb{R}^n\setminus \{0\}: \: V_b(f_i(x))<V_a(x).
\end{equation*}
\end{defn}
\vspace{-0.2cm}
We present, in what follows, the main stability result we will use in our stabilizing feedback design.
\begin{prop} \label{Thm:PCLFimpliesStability}
Consider a switched system~\eqref{eq:ClosedLoopSwitchedSystem} defined by $F = \{ f_1, \dots, f_M\}\subset \cC^0(\R^n, \R^n)$.  If there exists a path-complete Lyapunov function $(\cG,V)$ for $F$, then~\eqref{eq:ClosedLoopSwitchedSystem} is UGAS.
\end{prop}
The proof of Proposition~\ref{Thm:PCLFimpliesStability} can be found in~\cite{AhmJun:14, PEDJ:16, PhiAthAng}. This result establishes, given \emph{any} path-complete graph, sufficient conditions for stability of~\eqref{eq:ClosedLoopSwitchedSystem}. For a discussion of the conservatism and comparison among different graph structures for stability, see~\cite{DebDel22}.
In order to adapt path-complete Lyapunov stability techniques from \emph{stability} to \emph{stabilization} problem of~\eqref{eq:SwitchedSystemInput}, we need to refine Definition~\ref{defn:Path-Completeness}, introducing additional properties on the considered graphs, as defined in what follows. 

\begin{defn}[Complete graphs]\label{def:ComplateGraph}
A graph $\cG=(S,E)$ on $\langle M \rangle$ is \emph{complete} if for all $a \in S$, for all $i \in \langle M \rangle$, there exists at least one node $b \in S$ such that the edge $e=(a,b,i) \in E$.
\end{defn}
We note that any complete graph\footnote{\textcolor{black}{The considered notion of completeness, arising from automata theory, should not be confused with the notion of graph completeness imposing the existence of any possible edge, i.e. $E=S\times S\times \M$, more common in classic graph theory.}} is in particular path-complete; for a graphical visualization, we point out that the graphs in Figure~\ref{fig:GraphH} are complete. 
\textcolor{black}{
In what follows, we introduce a hierarchical family of complete graphs firstly defined in the seminal paper~\citep{DeBru46}, used in what follows as a tool for numerical verification of the proposed conditions.}

\begin{defn}[~\citep{DeBru46}]\label{defn:DeBRunjii}
Given $M \in \Zp, l\in \N$ the \emph{(primal) De Bruijn graph of order $l$ (on $\M$)} denoted by $\cH^l(M)=(S_{p,l},E_{p,l})$ is defined as follows: $S_{p,l}:=\M^{l}$ and, given any node $\wi=(i_1, \dots,i_{l})\in S_{p,l}$, we have $(\wi,\wj, h)\in E_{p,l}$ for every $\wj$ of the form $\wj=(h,i_1,\dots, i_{l-1})$, for any $h\in \M$.
\end{defn}
It can be seen that the graph $\cH^l(M)$ is complete $\forall\,M \in \Zp, l\in \N$; see Fig.~\ref{fig:GraphH} for a graphical representation.


\begin{figure*}[!t]
    \centering
    {\includegraphics[width=0.8\linewidth]{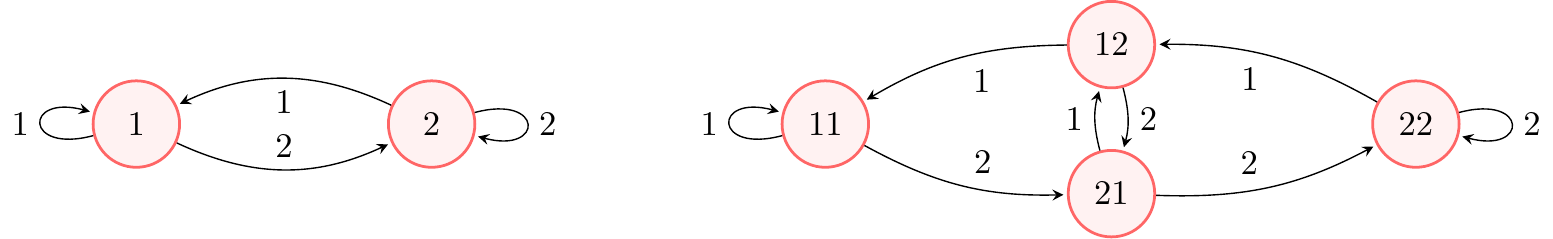}}
    \caption{The graphs $\mathcal{H}^1(2)$ and $\mathcal{H}^2(2)$, respectively the primal De Bruijn graphs of order $1$ and $2$ on $\M=\{1,2\}$.}
    \label{fig:GraphH}
\end{figure*}

\section{Robust Feedback via Path-complete Graphs}\label{Sec:RobFeedback}
\textcolor{black}{In this section, we provide an LMI formulation for the design of \emph{piecewise linear feedback} for the robust stabilization problem. The arising conditions will generalize existing approaches (e.g.~\citep{BlaMiaSav07,BlaMia03,HuBla10}), via the use of multiple, or more precisely, path-complete Lyapunov functions.} The theoretical developments are then illustrated by numerical examples.
\subsection{Piecewise Linear  Robust Feedback}\label{sec:RobustSubSec}
In this subsection we provide the first stabilization result, making use of the path-complete framework to provide \emph{min of quadratics} control Lyapunov functions, together with \emph{piecewise linear} robust feedback control map for system~\eqref{eq:SwitchedSystemInput}. 
\begin{prop}\label{Prop:RobustMachin}
Consider $M\in \Zp$ and a set $\cF=\{(A_i,B_i)\in \R^{n\times n}\times \R^{n\times m}\;\vert\;i\in \M\}$, and a \emph{complete} graph $\cG=(S,E)$ on $\M$. Suppose there exist $\{P_s\}_{s\in S}\subset \bS_+^{n\times n}$, $\{K_s\}_{s\in S}\subset\R^{m\times n}$ such that
\begin{equation}\label{eq:RobustConditionPiecewise}
(A_i+B_iK_a)^\top P_b (A_i+B_iK_a)-P_a\prec 0,\;\;\forall e=(a,b,i)\in E,
\end{equation}
then the map $\Phi(x):=K_{\gamma(x)}x$, $x\in \R^n$, robustly exponentially stabilizes system~\eqref{eq:SwitchedSystemInput}, where $\gamma:\R^n\to S$ is any function satisfying 
\begin{equation}\label{eq:argminDefinition}
\gamma(x)\in\argmin_{s\in S} \{x^\top P_sx\},\;\;\;\forall x\in \R^n,
\end{equation}
and for which $\Phi:\R^n\to \R^m$ is piecewise linear (recall Definition~\ref{defn:piecwiseLinear}).\footnote{If the node set $S$ is an ordered set, we can simply choose $\gamma(x)=\min_S(\argmin_{s\in S} \{x^\top P_sx\})$, where $\min_S(S')$ denotes the minimal element of $S'\subseteq S$, with respect to the~order~of~$S$.}
\end{prop}

\begin{proof}
The proof is inspired by proof of \citep[Proposition III.1]{PhiAthAng}.
Consider $V_s:\R^n\to \R$ defined by $V_s(x)=x^\top P_s x$, function $W(x)=\min_{s\in S}\{V_s(x)\}$ and $f_{is}:\R^n\to \R^n$ defined by $f_{is}(x)=(A_i+B_iK_s)x$, for any $s\in S$ and any $i\in \M$. Since $P_s\succ0$ for all $s\in S$, the function $W:\R^n\to \R$ is positive definite and radially unbounded. Condition~\eqref{eq:RobustConditionPiecewise} can be rewritten as
\begin{equation}\label{eq:technicalIneq}
V_b(f_{ia}(x))< V_a(x),\;\;\forall x\in \R^n\setminus\{0\},\;\forall e=(a,b,i)\in E.
\end{equation}
Consider any $x\in \R^n$, $x\neq 0$ and any $i\in \M$. Suppose $\gamma(x)=a\in \argmin_{s\in S} \{x^\top P_sx\}$, and thus $W(x)=V_a(x)$.
By completeness of $\cG$, there exists a $b\in S$ such that $e=(a,b,i)\in E$, and thus we have
\[
\begin{aligned}
W(f_{ia}(x))&=\min_{s\in S}\{V_s(f_{ia}(x))\}\leq V_b(f_{ia}(x))\\&< V_a(x)=W(x).
\end{aligned}
\]
We have thus proved that the function $W:\R^n\to \R$ is a  Lyapunov function for the closed-loop system
\[
x(k+1)=f_{\sigma(k)}(x(k)):=\left(A_{\sigma(k)}+B_{\sigma(k)}K_{\gamma(x(k))}\right)x(k)
\]
for any $\sigma:\N\to \M$. Since the closed loop system is homogeneous of degree $1$ (i.e. $f_j(\lambda x)=\lambda f_j(x)$ for any $j\in \M$, for any $x\in \R^n$, for any $\lambda\in \R$) it is in particular UES, concluding the proof.
\end{proof}

Thanks to the completeness of $\cG$ and the particular structure of~\eqref{eq:RobustConditionPiecewise}, we can provide, in what follows, \emph{necessary and sufficient} LMI conditions for~\eqref{eq:RobustConditionPiecewise}.

\begin{lemma}\label{lem:MinCOnditionsWithShur}
Conditions~\eqref{eq:RobustConditionPiecewise} are satisfied if and only if there exist $\{\xbar{P}_s\}_{s\in S}\subset \bS^{n\times n}$, $\{ \xbar{K}_s\}_{s\in S}\subset \R^{m\times n}$ such that
\begin{equation}\label{eq:LMIConditionRobust}
\begin{bmatrix}
\xbar{P}_b & (A_i\xbar{P}_a+B_i \xbar{K_a}) \\ \star & \xbar{P}_a
\end{bmatrix}\succ 0,\;\;\;\;\forall\;e=(a,b,i)\in E.
\end{equation}
 Matrices $\{P_s\}_{s\in S}\subset \bS_+^{n\times n}$, $\{K_s\}_{s\in S}\subset \R^{m\times n}$ satisfying \eqref{eq:RobustConditionPiecewise} are then given by defining $P_s=\xbar{P}_{s}^{-1}$ and $K_s=\xbar{K_s}\xbar{P}_{s}^{-1}$, for any $s\in S$.  
\end{lemma}
\begin{proof}\textit{Sufficiency:}
Suppose \eqref{eq:LMIConditionRobust} holds; it is easy to prove that it implies $\xbar{P}_s\succ0$ for all $s\in S$. For any $e=(a,b,i)\in E$ we pre- and post-multiply \eqref{eq:LMIConditionRobust} by $diag(\xbar{P}_b^{-1},\xbar{P}_a^{-1})$, obtaining
\[
\begin{bmatrix}
\xbar{P}_{b}^{-1} & \xbar{P}_{b}^{-1}(A_i+B_i \xbar{K_a}\xbar{P}_{a}^{-1}) \\ \star & \xbar{P}_{a}^{-1}
\end{bmatrix}\succ 0,\;\;\;\;\forall\;e=(a,b,i)\in E.
 \]
By defining changes of variables $P_s=\xbar{P}_s^{-1}$ and $K_s=\xbar{K}_s\xbar{P}_s^{-1}$ for any $s\in S$ we obtain
\begin{equation} \label{eq:robustconditionSchur}
\begin{bmatrix}
{P}_{b} & {P}_{b}(A_i+B_i {K_a}) \\ \star & {P}_{a}
\end{bmatrix}\succ 0,\;\;\;\;\forall\;e=(a,b,i)\in E,
 \end{equation}
 which, by Schur complement, is equivalent to
\begin{equation} \label{eq:robustconditionPieceWise2}
    {P}_{a} - (A_i+B_i {K_a})^{\top} {P}_{b} {P}_{b}^{-1} {P}_{b}(A_i+B_i {K_a})
\succ 0,
\end{equation}
 $\forall\;e=(a,b,i)\in E$, leading to \eqref{eq:RobustConditionPiecewise}. 
 
 \textit{Necessity:} Suppose that \eqref{eq:RobustConditionPiecewise} holds. Rewrite it as \eqref{eq:robustconditionPieceWise2} and apply Schur complement again to obtain \eqref{eq:robustconditionSchur}. Next, apply a congruence transformation with $diag({P}_b^{-1},{P}_a^{-1})$ followed by the same changes of variables as before to obtain \eqref{eq:LMIConditionRobust}. Therefore, feasibility of \eqref{eq:RobustConditionPiecewise} implies satisfaction of \eqref{eq:LMIConditionRobust} with $\xbar{P}_s=P_s^{-1}$ and $K_s P_s^{-1}=\xbar{K}_s$.
\end{proof}

Proposition~\ref{Prop:RobustMachin} and the subsequent LMI characterization in Lemma~\ref{lem:MinCOnditionsWithShur} are stated for generic complete graphs. \textcolor{black}{In Definition~\ref{defn:DeBRunjii}, we introduced the class of De-Bruijn graphs, denoted by $\cH^l(M)$. In the following statement, as a corollary, we explicitly write the conditions of Lemma~\ref{lem:MinCOnditionsWithShur} associated to $\cH^l(M)$, proving the numerical scheme used in the following sections.}

\begin{cor}[De Bruijn conditions, Robust Case]\label{cor:DeBrunjiiRObust}
Consider $M\in \Zp$, a set $\cF=\{(A_i,B_i)\in \R^{n\times n}\times \R^{n\times m}\;\vert\;i\in \M\}$, and any $l\in \N$. Suppose there exist $\{\xbar{P}_{\wi} \}_{\wi\in \M^{l}}\subset \bS^{n\times n}$, $\{\xbar{K}_{\wi}\,\}_{\wi \in \M^{l}}\subset \R^{m\times n}$ such that, $\forall\wi=(i_1,\dots, i_{l})\in \M^{l}$ and $\forall \,h\in \M$, the inequalities
\begin{equation}\label{eq:LMIConditionRobustDeBrunjii}
\begin{bmatrix}
\xbar{P}_{(h,\wi^-)} & (A_h\xbar{P}_{\wi}+B_h \xbar{K}_{\wi}) \\ \star & \xbar{P}_{\wi}\
\end{bmatrix}\succ 0
\end{equation}
hold, where $\wi^-:=(i_1,i_2,\dots, i_{l-1})\in \M^{l-1}$.
Then the piecewise linear map $\Phi(x):={K}_{\gamma(x)}x$, $x\in \R^n$, where $K_{\wi}= \xbar{K}_{\wi} \xbar{P}_{\wi}^{-1}$, robustly exponentially stabilizes system~\eqref{eq:SwitchedSystemInput} where $\gamma:\R^n\to \M^{l}$ is a  function satisfying 
\[
\gamma(x)\in\argmin_{\wi\in \M^{l}} \{x^\top {P}_{\wi}x\}\;,\;\;\forall x\in \R^n,
\]
with ${P}_{\wi} = \xbar{P}^{-1}_{\wi}$ and for which $\Phi:\R^n\to \R^m$ is piecewise linear. Moreover, the function $W(x):=\min_{\wi \in \M^{l}} \{x^\top P_{\wi} x\}$ is a Lyapunov function for the closed loop system.
\end{cor}
\begin{proof}
Conditions in~\eqref{eq:LMIConditionRobustDeBrunjii} are the specification  of conditions~\eqref{eq:LMIConditionRobust}  when considering the graph $\cH^l(M)$ introduced in~Definition~\ref{defn:DeBRunjii}; by Proposition~\ref{Prop:RobustMachin} we conclude.
\end{proof}
\begin{rem}[Robust Stabilization of LPVs]\label{rem:RobustLPV}
It is important to point out that the stabilization techniques proposed in Proposition~\ref{Prop:RobustMachin} and then numerically tackled in Lemma~\ref{lem:MinCOnditionsWithShur} and Corollary~\ref{cor:DeBrunjiiRObust} for~\eqref{eq:SwitchedSystemInput} are effective also for the more general class of LPV systems as in~\eqref{eq:DiscTimeLPV}. 
Indeed, suppose that a piecewise linear control $\Phi:\R^n\to \R^m$ as in Proposition~\ref{Prop:RobustMachin} is given, it can be shown, by a convexity argument, that the \emph{convex difference inclusion}
$
x(k+1)\in \co\left\{ A_ix(k)+B_i\Phi(x(k))\;\vert\;i\in \M \right\}
$
is exponentially stable, for a direct proof, we refer to~\citep[Proposition 2]{BlaMiaSav07}. This, in particular, implies the exponential stability of~\eqref{eq:DiscTimeLPV}.
\hfill $\triangle$
\end{rem}

Concluding this section, we briefly illustrate an open direction of research. 
Corollary~\ref{cor:DeBrunjiiRObust} provides sufficient conditions for the robust feedback stabilizability of~\eqref{eq:SwitchedSystemInput}. It can be seen that, if the conditions related to the graph $\cH^l(M)$ are feasible, then also the conditions related to graphs $\cH^{l'}(M)$ are feasible, for any $l'\geq l$. Thus, by increasing the order $l\in \N$ one can reduce the conservatism of the proposed conditions (as we will illustrate in the next subsection). Such a hierarchy is known to provide non-conservative conditions (by increasing the order $l\in \N$) for the stability problem, see~\cite{AhmJun:14}. On the other hand, the question of whether this is also  the case for the robust stabilization problem studied here is still open. More precisely, future research will investigate the following claim:
\begin{conjecture}
Consider $M\in \Zp$ and a set $\cF=\{(A_i,B_i)\in \R^{n\times n}\times \R^{n\times m}\;\vert\;i\in \M\}$, and suppose that system~\eqref{eq:SwitchedSystemInput} is robust feedback stabilizable, in the sense of Definition~\ref{defn:StabNot}. Then, there exists a $l\in \N$ (large enough) for which the conditions in Corollary~\ref{cor:DeBrunjiiRObust} are feasible.
\end{conjecture}

\subsection{Numerical Examples: Robust Case}
In this subsection, we provide some numerical examples in order to illustrate the results presented in Subsection~\ref{sec:RobustSubSec}, and the advantages of the proposed techniques with respect to existing literature.

\subsubsection{Example 1}~\label{ex:1}
Consider system~\eqref{eq:IntroSystem} with matrices 
{
\begin{equation*}
    A_1 = \begin{bmatrix}
        0.1 & 0.9 \\ 0 & 0.1
    \end{bmatrix},\; A_2 = \begin{bmatrix}
        \alpha & 0 \\ 1 & \alpha
    \end{bmatrix},\, B_1 = \begin{bmatrix}
        0 \\ 0
    \end{bmatrix},\, B_2 = \begin{bmatrix}
        \alpha \\ -1
    \end{bmatrix}. 
\end{equation*}}
where $\alpha>0$ is a free parameter. The goal is to find the maximal $\alpha$ for which conditions of Corollary~\ref{cor:DeBrunjiiRObust} are feasible, thus providing feedback gains for which the closed-loop system is UES. The stabilization of this system was also studied in~\cite[Example~2]{Lee06}, where the authors design ``finite-path dependent state feedback controllers'', that is, state-feedback controllers that depend on the history of the switching signal~$\sigma$. The maximum value of $\alpha$ for which stability is achieved in~\cite{Lee06} is $\alpha = 0.6667$ by using controllers with memory. If no memory is employed, the maximum obtained value in~\cite{Lee06} is $\alpha = 0.6173$. 
\begin{table}[t!]
\centering
\caption{Relation between maximum $\alpha$ and graph order $l$ for Example~2.}
\begin{tabular}{l|lllll}
                           &\cellcolor{gray!20}   $l=0$ & \cellcolor{gray!20} $l=1$ & \cellcolor{gray!20} $l=2$ &\cellcolor{gray!20}  $l=3$ & \cellcolor{gray!20} $l=5$  \\ 
\hline
max $\alpha$ & 0.6646  & 0.6684  & 0.6856  & 0.6984   & 0.7032 
\end{tabular}\label{tab:exampleRobust}
\end{table}
By using the conditions of Corollary~\ref{cor:DeBrunjiiRObust}, we obtain that the maximum value obtained is $\alpha = 0.5872$ for the case of the De Bruijn graph of order $l=0$, and $\alpha = 0.6667$ for all $l \geq 1$. Thus, with the strategy developed here, we can robustly stabilize the system for an $\alpha$ that is as large as the maximum $\alpha$ obtained in~\cite{Lee06} by using memory-dependent controllers, which require knowledge/measurement of switching signals. In conclusion, our robust feedback strategy is as conservative as the memory-dependent strategy in~\cite{Lee06}. This was expected since the underlying conditions of~\cite{Lee06} are substantially equivalent to the ones in~Corollary~\ref{cor:DeBrunjiiRObust}. Our graph approach allows us to provide a feedback control map in a closed form as defined in Proposition~\ref{Prop:RobustMachin}, without the necessity of storing the past active modes.

\begin{figure*}[b!]
\begin{align}\label{eq:Pmatrices}
\begin{split}
    &P_{11} = \begin{bmatrix} \phantom{-}1.7981  &  \phantom{-}0.5562 \\
    \phantom{-}0.5562  &  \phantom{-}0.9599 \end{bmatrix},
    P_{12} = \begin{bmatrix}  \phantom{-}1.6821   & \phantom{-}0.4795 \\
    \phantom{-}0.4795  &  \phantom{-}0.8672 \end{bmatrix},
    P_{13} = \begin{bmatrix}  \phantom{-}1.9114   & \phantom{-}0.5837 \\
   \phantom{-}0.5837 &  \phantom{-}0.7375 \end{bmatrix},\\ 
    &P_{21} = \begin{bmatrix}  \phantom{-}1.6331  & -0.6584\\
   -0.6584 &  \phantom{-}1.1943
 \end{bmatrix},
    P_{22} = \begin{bmatrix} \phantom{-}1.4815 &  -0.4356\\
   -0.4356  & \phantom{-}0.9449 \end{bmatrix},
    P_{23} = \begin{bmatrix} \phantom{-}1.9377  & -0.3945\\
   -0.3945 &  \phantom{-}0.5639 \end{bmatrix},\\
    &P_{31} = \begin{bmatrix} \phantom{-}1.7896 &  -0.4346\\
   -0.4346 &  \phantom{-}0.7507 \end{bmatrix},
    P_{32} = \begin{bmatrix} 29.8984 &  -6.6345\\
   -6.6345  & \phantom{-}1.8047 \end{bmatrix},
    P_{33} = \begin{bmatrix} 24.6378 &  -5.4361\\
   -5.4361   & \phantom{-}1.5320 \end{bmatrix}.
   \end{split}
\end{align}
\vspace{-0.5cm}
\begin{align}\label{eq:Kmatrices}
\begin{split}
    &K_{11} = \begin{bmatrix} -0.4816 &  -0.3583 \end{bmatrix},
    K_{12} = \begin{bmatrix} -0.4660  & -0.3460 \end{bmatrix},
    K_{13} = \begin{bmatrix} -0.4861 &  -0.3628 \end{bmatrix},\\
    &K_{21} = \begin{bmatrix} -0.3591 &  -0.5802 \end{bmatrix},
    K_{22} = \begin{bmatrix} -0.3652 &  -0.5817 \end{bmatrix},
    K_{23} = \begin{bmatrix} -0.4011 &  -0.6044 \end{bmatrix},\\
    &K_{31} = \begin{bmatrix} -0.4124 &  -0.5761 \end{bmatrix},
    K_{32} = \begin{bmatrix} -1.0548 &  -0.4989 \end{bmatrix},
    K_{33} = \begin{bmatrix} -0.9970 &  -0.5116 \end{bmatrix}.
    \end{split}
\end{align}
\end{figure*}

\subsubsection{Example 2}
We modify the system from Example~$1$ by setting $B_1 = [
    1 \;\, 0]^{\top}$. In this case, the maximum $\alpha$ obtained for which the system can be stabilized as a function of the order $l$ of the chosen De Bruijn graph is illustrated in Table~\ref{tab:exampleRobust}. As demonstrated by the table, an increase in $l$ tends to improve the robustness to the uncertainty $\alpha$. 

Given $l=2$ and $\alpha=0.6856$, for any $\wi\in \langle 2\rangle ^2$, we obtain the following Lyapunov matrices $P_{\wi}$ and controllers $K_{\wi}$: 
\begin{align*}
   P_{11}& = \begin{bmatrix}   1.9133  &  0.2510 \\
    0.2510 &   1.7315 \end{bmatrix},
    ~P_{12} = \begin{bmatrix} 1.6753  &  0.3818 \\
    0.3818 &   1.1150\end{bmatrix},\\
    ~P_{21}& = \begin{bmatrix} 3.2896  & -0.0436 \\
   -0.0436 &   0.5803 \end{bmatrix},
    ~P_{22} = \begin{bmatrix} 2.0241   & 0.4344 \\
    0.4344  &  0.5322 \end{bmatrix}, \\ 
    K_{11} &= \begin{bmatrix} -0.3162 & -0.3347 \end{bmatrix},
    ~K_{12} = \begin{bmatrix} -0.3342 & -0.3596 \end{bmatrix}, \\ 
    ~K_{21} &= \begin{bmatrix} 0.6279 & -0.4343 \end{bmatrix},
    ~K_{22} = \begin{bmatrix} -0.0714 & -0.4110 \end{bmatrix}.
\end{align*}

\subsubsection{Example 3}
This example illustrates the relationship between the De Bruijn graph order $l$ and the closed-loop performance, given by the minimum $\gamma \in [0,1)$ that satisfies $W(x(k+1)) \leq \gamma^2\, W(x)$, where $W$ is the Lyapunov function defined by $W(x):=\min_{\wi\in \M^{l}}\{x^\top P_{\wi} x\}$ given in Corollary~\ref{cor:DeBrunjiiRObust}. This $\gamma$ represents an upper bound for the exponential decay rate of the solutions of the closed loop, recall Definition~\ref{Def:UGASandUES}. The conditions of Corollary~\ref{cor:DeBrunjiiRObust} can be straightforwardly modified to tackle this problem by multiplying the $(2,2)$-block of the matrix in~\eqref{eq:LMIConditionRobustDeBrunjii} by $\gamma^2$; then a line search is performed to minimize the parameter $\gamma$. Consider the system defined by the following matrices:
{
\begin{gather*}
    A_1 = \begin{bmatrix}  \phantom{-}0 & \;1 \\ -1 & \;0  \end{bmatrix},  A_2 = \begin{bmatrix}  -1 & \;0 \\ \phantom{-}0 & \;-0.95  \end{bmatrix},  B_1 = B_2 = \begin{bmatrix}  1 \\ 0  \end{bmatrix}.
\end{gather*}}
In Table~\ref{tab:exampleDecayRate}, we summarize the relationship between the graph order $l$ and minimum decay rate $\gamma$ achieved. 
Moreover, Fig.~\ref{fig:sublevelsets} illustrates the sub-level set of the Lyapunov function $W$ given by $\mathcal{E}_{_W}(x)=\{x \in \R^n : W(x) \leq 1\}$ and its relationship with the graph order $l$. The level sets of $W(x)$ are marked in red in the figure, while the sub-level sets defined of each function $V_{s}(x)$ are illustrated in black. 
\begin{table}[t!]
\centering
\caption{Relation between minimum upper bound on the decay rate $\gamma$ and graph order $l$ for Example~3.}
\begin{tabular}{l|lllll}
                           &\cellcolor{gray!20}   $l=0$ & \cellcolor{gray!20} $l=2$ & \cellcolor{gray!20} $l=4$ &\cellcolor{gray!20}  $l=6$ & \cellcolor{gray!20} $l=10$  \\ 
\hline
min $\gamma$ & 0.9638  & 0.9612  & 0.9586  & 0.9571   & 0.9555  
\end{tabular}\label{tab:exampleDecayRate}
\end{table}

\begin{figure*}[t!]
\centering
\begin{subfigure}{.3\linewidth}
    \centering
    {\includegraphics[width=\linewidth,angle=0]{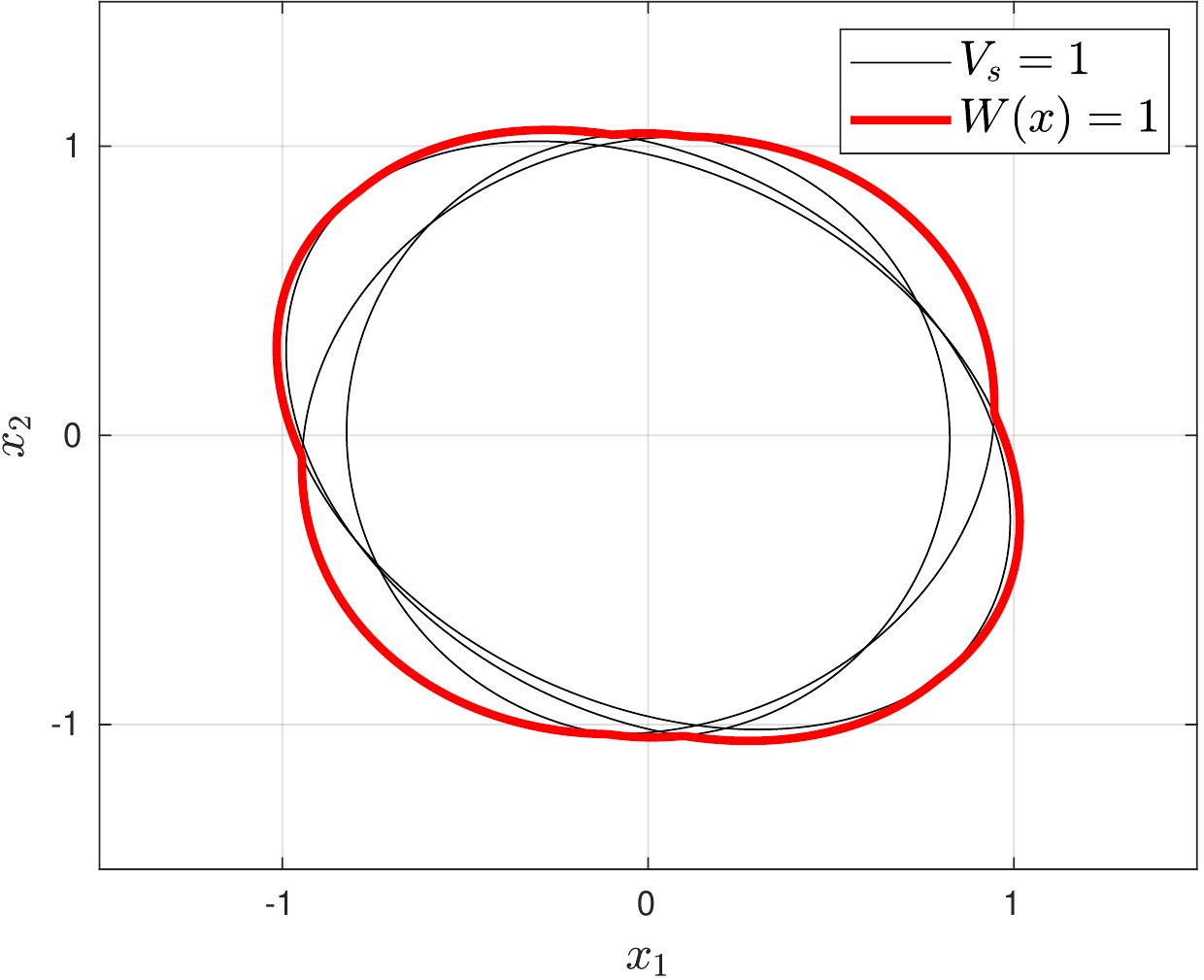}}
    \caption{$l=2$}\label{fig:a}
\end{subfigure}
    \hfill
\begin{subfigure}{.3\linewidth}
    \centering
   {\includegraphics[width=\linewidth,angle=0]{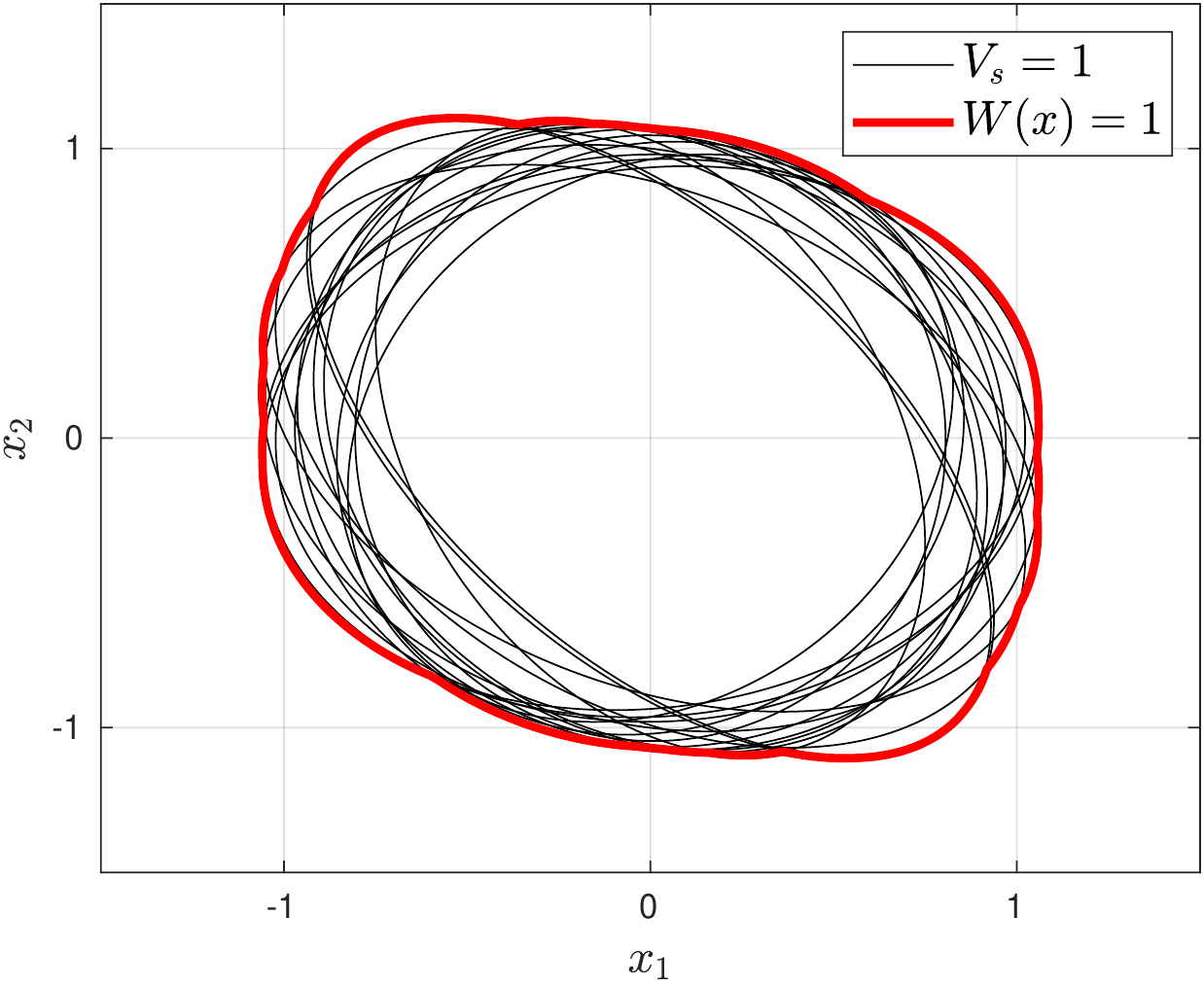}}
    \caption{$l=4$}\label{fig:b}
\end{subfigure}
   \hfill
\begin{subfigure}{.3\linewidth}
    \centering
    {\includegraphics[width=\linewidth,angle=0]{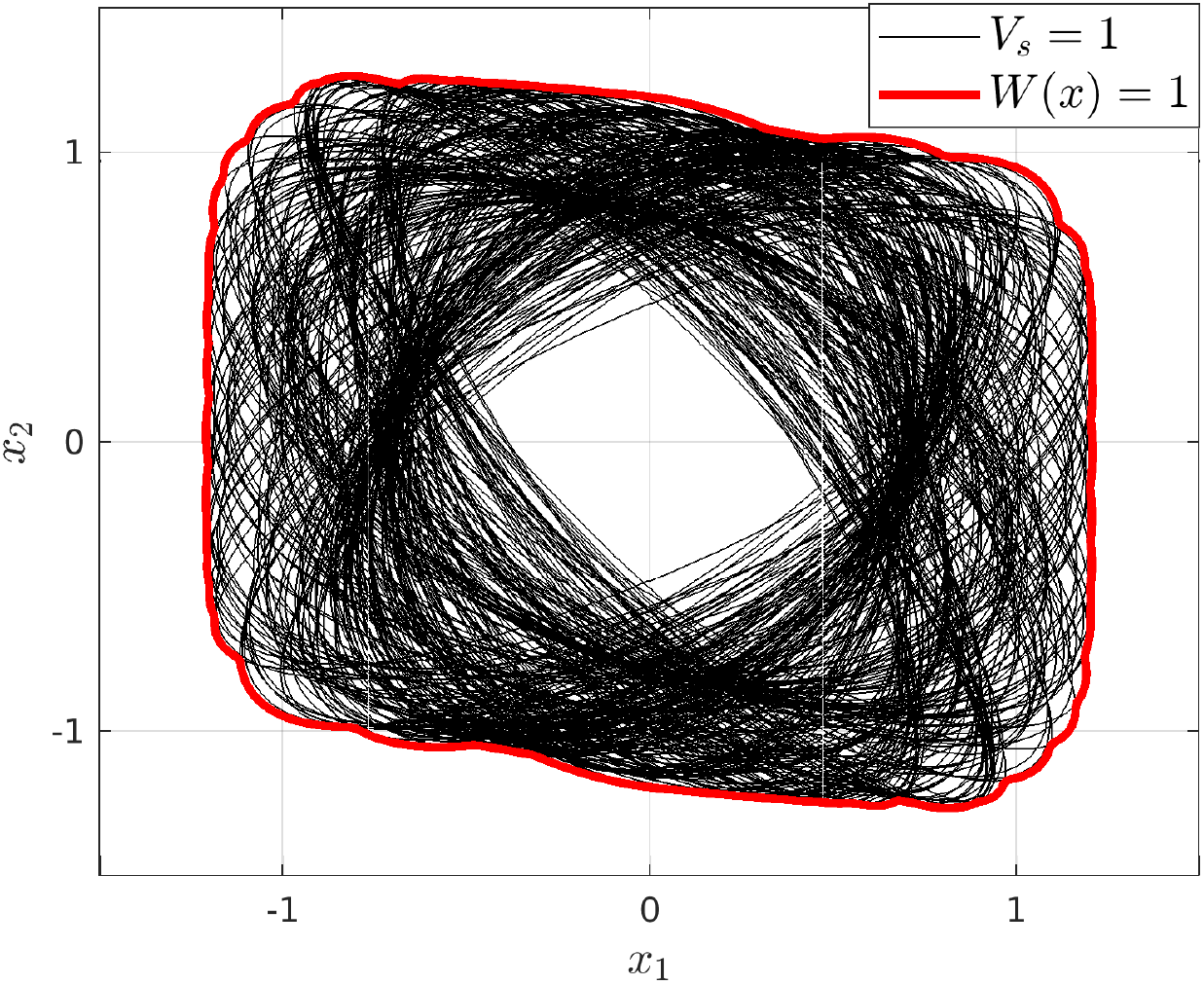}}
    \caption{$l=8$}\label{fig:c}
\end{subfigure}
\RawCaption{\caption{Level sets for the obtained Lyapunov functions $W(x)$ for different graph orders $l$ in Example~3.}
\label{fig:sublevelsets}}
\end{figure*}

\subsubsection{Example 4}
Consider a system defined by $B_1 = B_2 = B_3 = [1~\;0]^\top$,
{
\begin{align*}
    A_1 = \begin{bmatrix} 1 & 0.2 \\ 0 & 0.5 \end{bmatrix},
    A_2 = \begin{bmatrix} 1.1 & 0.2 \\ 0  & -0.5 \end{bmatrix},
    A_3 = \begin{bmatrix} 0.5 & 0.8 \\ 1.1 & 0.5 \end{bmatrix}.
\end{align*}}
This example illustrates the geometric intuition behind the \textit{min of quadratics} strategy that selects the linear controller at each instant of time. By choosing $l=2$, the conditions of Corollary~\ref{cor:DeBrunjiiRObust} are feasible and thus we find the Lyapunov function $W(x)=\min_{\wi \in \langle 3 \rangle^2}\{x^\top P_{\wi}x\}$, which guarantees a decay rate of $\gamma=0.7536$, defined by the matrices $P_{\wi}$, for $\wi\in \langle 3 \rangle^2$ in~\eqref{eq:Pmatrices}. The obtained controller gains are given in~\eqref{eq:Kmatrices}.

In Fig.~\ref{fig:cones} we have plotted the 1-level set of $W(x)$, and one closed-loop trajectory, starting at the initial condition $x_0 = [0.4 \;\, 1.4]^{\top}$ and following a periodic switching sequence $\sigma = \{1,2,3,1,2,3\dots\}$. As previously illustrated, the level set of $W(x)$ correspond to the union of the level sets of the multiple $V_{\wi}$, $\wi\in  \langle 3 \rangle^2$. Recalling Definition~\ref{defn:piecwiseLinear}, the cones defining the partition associated to the piecewise linear feedback map $\Phi$ in Proposition~\ref{Prop:RobustMachin} are defined by the \emph{argmin} among a set of quadratic functions, see~\eqref{eq:argminDefinition}. That is, in such cones, the minimum is attained for the same $s\in S= \langle 3 \rangle^2$, i.e. $W(x)=V_s(x)$ is achieved. For example, the black cone in Fig.~\ref{fig:cones} represents the cone where the value  of $W(x)$ coincides with $V_{33}(x)$. Whenever the state $x$ belongs to this cone, the controller gain $K_{33}$ is activated.  For the plotted trajectory, only three controllers have been active, namely $K_{13}$, $K_{21}$, and $K_{33}$. However, the state can fall within different cones where other controller gains will be activated for different initial conditions and switching signals.
In conclusion, from a geometric perspective, an increase in the graph order $l$ tends not only to \textit{refine} the level set of $W(x)$ but also to produce more cones so that different controller gains better adapted to certain regions of the state space are used, thus producing an improved control law.

\begin{figure}
    \centering
    {\includegraphics[width=0.5\linewidth]{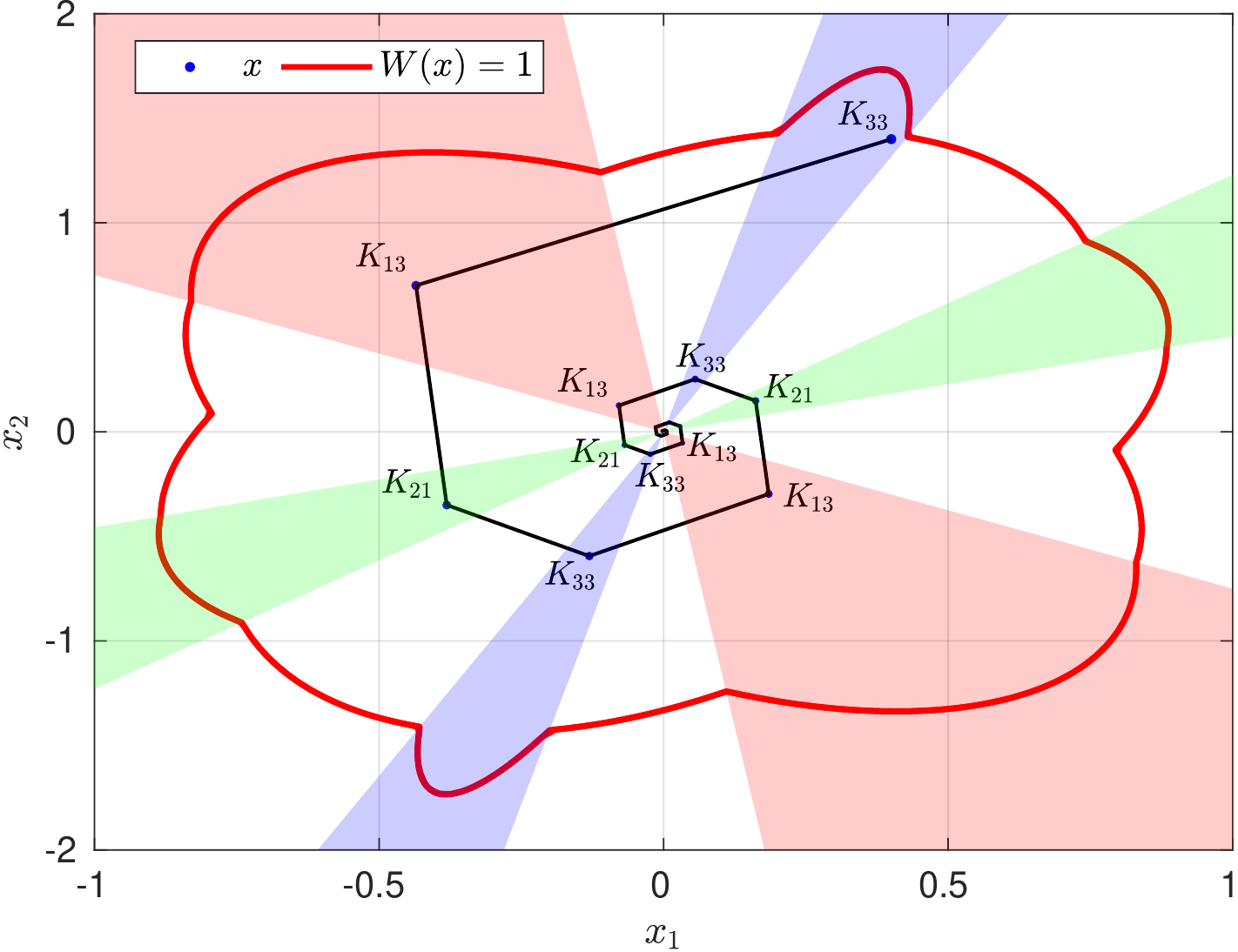}}
    \caption{Level set of~$W(x)$, one trajectory of the closed-loop system, and the controller used at each point.}
    \label{fig:cones}
\end{figure}

\section{Mode-Dependent Case}\label{Sec:ModeDep}
In this section we provide our main stabilization results in the mode-dependent case.
\subsection{ Piecewise Linear  Mode-Dependent Feedbacks}\label{sec:ModeDepWithMin}
In this subsection, adapting the proof technique of Proposition~\ref{Prop:RobustMachin}, we propose conditions depending on a  graph structure, leading to piecewise-linear mode-dependent feedback gains.
To this aim, we require an additional property on the underlying graphs, and we thus introduce the following definition.

\begin{defn}[Deterministic Graph]\label{defn:DetGraph}
A graph $\cG=(S,E)$ on $\M$ is said to be \emph{deterministic} if, for all $a\in S$ and $i\in \M$ there exists at most one $b\in \M$ such that $e=(a,b,i)\in E$.
\end{defn}

\begin{prop}\label{Prop:COmpleteModeDependent}
Consider $M\in \Zp$, a set $\cF=\{(A_i,B_i)\in \R^{n\times n}\times \R^{n\times m}\;\vert\;i\in \M\}$ and a \emph{complete and deterministic} graph $\cG=(S,E)$ on $\M$. Suppose there exist $\{P_s\}_{s\in S}\subset \bS_+^{n\times n}$, $\{K_{s,j}\}_{(s,j)\in S\times \M}\subset \R^{m\times n}$ such that
\begin{equation}\label{eq:RobustConditionPiecewiseModeDep}
(A_i+B_iK_{a,i})^\top P_b (A_i+B_iK_{a,i})-P_a\prec 0,
\end{equation}
$\forall e=(a,b,i)\in E$. Then, the piecewise linear maps $\Phi_i(x):=K_{\gamma(x),i}\,x$, for $i\in \M$,  exponentially stabilize system~\eqref{eq:SwitchedSystemInput}(in the mode-dependent sense of Definition~\ref{defn:StabNot}) where $\gamma:\R^n\to S$ is any function satisfying
\[
\gamma(x)\in\argmin_{s\in S} \{x^\top P_sx\},\;\;\;\forall\;x\in \R^n,
\]
and for which the functions $\Phi_i:\R^n\to \R^m$ are piecewise linear.
\end{prop}
\begin{proof}
First of all, since $\cG$ is complete and deterministic, for any $a\in S$ and any $i\in \M$, there exists a \emph{unique} $b\in S$ such that $e=(a,b,i)\in E$ and thus the notation $K_{s,i}$ is well posed; then the proof fundamentally follows the structure of the proof of Proposition~\ref{Prop:RobustMachin}.
Define $V_s:\R^n\to \R$ by $V_s(x)=x^\top P_s x$, function $W(x)=\min_{s\in S}\{V_s(x)\}$ and $f_{is}:\R^n\to \R^n$ by $f_{is}(x)=(A_i+B_iK_{s,i})x$, for any $s\in S$ and any $i\in \M$. Again,~\eqref{eq:RobustConditionPiecewiseModeDep} implies that $W:\R^n\to \R$ is positive definite and radially unbounded, since $P_s\succ0$ for all $s\in S$. With this notation, condition~\eqref{eq:RobustConditionPiecewiseModeDep} implies inequalities~\eqref{eq:technicalIneq}, and thus, following the reasoning of proof of Proposition~\ref{Prop:RobustMachin}, we  prove that the function $W:\R^n\to \R$ is a Lyapunov function for the closed-loop system
$
x(k+1)=\left(A_{\sigma(k)}+B_{\sigma(k)}K_{\gamma(x(k)),\sigma(k)}\right)x(k)
$
for any $\sigma:\N\to \M$, concluding the proof.
\end{proof}

Next, we present \emph{necessary and sufficient} LMI conditions ensuring~\eqref{eq:RobustConditionPiecewiseModeDep}. 

\begin{lemma}\label{lem:LMIConditionsWithShurPiecewiseModeDep}
Conditions~\eqref{eq:RobustConditionPiecewiseModeDep} are satisfied if and only if there exist $\{\xbar{P}_s\}_{s\in S}\subset \bS^{n\times n}$, $\{ \xbar{K}_{s,i}\}_{s\in S,i \in \M}\subset \R^{m\times n}$ such that the LMIs
\begin{equation}\label{eq:LMIConditionPiecewiseModeDep}
\begin{bmatrix}
\xbar{P}_b & (A_i\xbar{P}_a+B_i \xbar{K}_{a,i}) \\ \star & \xbar{P}_a
\end{bmatrix}\succ 0,\;\;\;\;\forall\;e=(a,b,i)\in E,
\end{equation}
are feasible. Matrices $\{P_s\}_{s\in S}\subset \bS^{n\times n}$, $\{K_s\}_{s\in S}\in \R^{m\times n}$ satisfying \eqref{eq:RobustConditionPiecewise} are then given by defining $P_s=\xbar{P}_{s}^{-1}$ and $K_{s,i}=\xbar{K}_{s,i}\xbar{P}_{s}^{-1}$.  
\end{lemma}
\begin{proof}
The proof follows the same steps of the proof of Lemma~\ref{lem:MinCOnditionsWithShur} and is thus omitted.
\end{proof}

Similar to the Section~\ref{Sec:RobFeedback}, we are going to introduce below the equivalent LMI conditions for the class of De-Bruijn graphs of order $l\in \M$, denoted by $\cH^l(M)$, which are complete and deterministic, recall Definition~\ref{defn:DeBRunjii}.

\begin{cor}[De Bruijn: Mode-Dependent Case]\label{cor:DeBrunjiiGain} 
Consider $M\in \Zp$, a set $\cF=\{(A_i,B_i)\in \R^{n\times n}\times \R^{n\times m}\;\vert\;i\in \M\}$ and any $l\in \N$. Suppose there exist $\{\xbar{P}_{\wi} \}_{\wi\in \M^{l}}\subset \bS^{n\times n}$, $\{\xbar{K}_{\wi,h}\,\}_{\wi \in \M^{l}, h \in \M}\in \R^{m\times n}$ such that, $\forall\wi=(i_1,\dots, i_{l})\in \M^{l}$ and $\forall \,h\in \M$, the inequalities
\begin{equation}\label{eq:LMIConditionModeDependentBrunjii}
\begin{bmatrix}
\xbar{P}_{(h,\wi^-)} & (A_h\xbar{P}_{\wi}+B_h \xbar{K}_{\wi,h}) \\ \star & \xbar{P}_{\wi}\
\end{bmatrix}\succ 0,
\end{equation}
hold, where $\wi^-=(i_1,\dots, i_{l-1})\in \M^{l-1}$.
Then, the feedbacks maps $\Phi_h(x):={K}_{\gamma(x),\,h}\,x$, $x\in \R^n$, $h \in \M$, where $K_{\wi,h}= \xbar{K}_{\wi,h} \xbar{P}_{\wi}^{-1}$ exponentially stabilize system~\eqref{eq:SwitchedSystemInput} where $\gamma:\R^n\to \M^{l}$ satisfies
\begin{gather*}
    \gamma(x)\in\argmin_{\wi\in \M^{l}} \{x^\top {P}_{\wi}x\}
\end{gather*}
with ${P}_{\wi} = \xbar{P}^{-1}_{\wi}$ and for which the functions $\Phi_h:\R^n\to \R^m$ are piecewise linear. Moreover, $W(x):=\min_{\wi \in \M^{l}}\{x^\top P_{\wi} x\}$ is a Lyapunov function for the closed-loop system~\eqref{eq:modeDepSys}. 
\end{cor}
\begin{proof}
The proof is obtained by applying Proposition~\ref{Prop:COmpleteModeDependent} and Lemma~\ref{lem:LMIConditionsWithShurPiecewiseModeDep} to the graph $\cH^l(M)$ introduced in Definition~\ref{defn:DeBRunjii}.
 \end{proof}
\begin{rem}\textbf{\emph{(Relations with Gain-Scheduling Stabilization of LPV)}}\\
As highlighted in Remark~\ref{rem:RobustLPV}, the \emph{robust} stabilization problem for switched systems studied in Section~\ref{Sec:RobFeedback} is equivalent to robust stabilizability for LPV systems~\eqref{eq:DiscTimeLPV}. Unfortunately, this is not the case for the mode-dependent stabilization problem studied in this section. Indeed, in~\cite[Example 4.1]{BlaMiaSav07} it is shown that, given $M\in \Zp$ and a set $\cF=\{(A_i,B_i)\in \R^{n\times n}\times \R^{n\times m}\;\vert\;i\in \M\}$, mode-dependent stabilizability of~\eqref{eq:SwitchedSystemInput} is strictly weaker than mode-dependent (in this literature, a.k.a. gain-scheduling) stabilizability of~\eqref{eq:DiscTimeLPV}. Then, Proposition~\ref{Prop:COmpleteModeDependent} cannot be applied directly for~\eqref{eq:DiscTimeLPV}. 
On the other hand, in the particular case $B_1=\dots=B_M\in \R^{n\times m}$ i.e. when the input matrix does \emph{not} depend on the mode, mode-dependent stabilizability of~\eqref{eq:DiscTimeLPV} and mode-dependent stabilizability of~\eqref{eq:DiscTimeLPV} are indeed equivalent (see~\cite[Proposition 2]{BlaMiaSav07}), and thus Proposition~\ref{Prop:COmpleteModeDependent} can be applied to the (more general) class of LPV systems.\hfill $\triangle$
\end{rem}

\begin{rem}\textbf{\emph{(Comparison with Existing Results)}}\\
The LMIs conditions presented in Corollary~\ref{cor:DeBrunjiiGain} already appeared in~\cite{LeeKha09}, in a slightly different setting. Indeed, the authors of~\cite{LeeKha09} related the feasibility of the conditions with the existence of controllers relying on the knowledge of \emph{past switching} sequences, and thus it is required that the controller store the past active modes. On the other hand, we generalize the results in~\cite{LeeKha09} through Proposition~\ref{Prop:COmpleteModeDependent} and Lemma~\ref{lem:LMIConditionsWithShurPiecewiseModeDep}, where we rely only on the assumption that a graph is complete and deterministic to obtain LMI conditions assuring the existence of stabilizing piecewise linear controllers. Thus, more general graphs other than De Bruijn ones employed in Corollary~\ref{cor:DeBrunjiiGain} can be used too, potentially leading to more efficient results, given a particular system. On top of recovering the results of~\cite{LeeKha09} in a more general setting, the graph-theory approach used here allows us to provide piecewise linear controllers (along with min-of-quadratic Lyapunov functions) in a closed form without the necessity of observing and storing the past active modes. For a more formal discussion on the relations between graph-based stability conditions and conditions relying on past/future switching sequences, we refer to~\citep{DelRosJun22}.
We point out that in~\cite{LeeKha09} a negative result is proved: there exist mode-dependent feedback stabilizable systems of the form~\eqref{eq:modeDepSys} for which there does not exist a $l\in \N$ large enough such that the conditions in Corollary~\ref{cor:DeBrunjiiGain} are feasible.  We do not report the proof here; it can be found~\citep[Theorem 27 \& Example 28]{LeeKha09} and references therein.~\hfill $\triangle$
\end{rem}

\subsection{Example 5}
Consider a four-mode, third-order system studied in~\cite{BlaMia03} defined by $B_1 = B_2 = B_3 = B_4 = [0 \;\, 0 \;\, 0.3]^{\top}$ and
{
\begin{align*}
    &A_1 = \begin{bmatrix} 1 &.25 & 0 \\
    0.25 & 1 &-0.2 \\
    0 &0 &-0.16 \end{bmatrix},
    A_2 = \begin{bmatrix} 1 &.25 &0 \\
    0.25 &1 &-0.05 \\
    0 & 0 & 0.16\end{bmatrix},\\
    &A_3 = \begin{bmatrix} 1 & 0.32 & 0 \\
    0.32 & 1 & -0.05 \\
    0 & 0 & -0.16\end{bmatrix},
    A_4 = \begin{bmatrix}
        0.32 & 0 & 0 \\
    0.32 & 1  &-0.2 \\
    0 & 0& 0.16
    \end{bmatrix}.
\end{align*}}
In~\cite{BlaMia03}, by using techniques based on polyhedral Lyapunov functions, a stabilizing mode-dependent controller is provided,  also providing an upper-bound on the decay rate for the closed loop, given by $\overline \gamma=0.96$. Here we apply Corollary~\ref{cor:DeBrunjiiGain} with different graph orders $l$ to relate to the minimum upper-bound on the decay rate by using a similar modification to the LMIs~\eqref{eq:LMIConditionModeDependentBrunjii} as in Example~3. The results are illustrated in Table~\ref{tab:exampleDecayRateMD}, showing that an increase in the order graph $l$ allows smaller values of the decay rate $\gamma$, i.e., improving the convergence speed of the arising closed-loop. The decrease in theoretical conservatism comes at the cost of increasing the computational complexity, as an increase in $l$ also generates an increase in the number of variables of the LMIs in Corollary~\ref{cor:DeBrunjiiGain}. 


\begin{table}[h!]
\centering
\caption{Relation between minimum upper-bound on the decay rate $\gamma$ and graph order $l$ for Example~5.}
\begin{tabular}{l|lllll}
                           &\cellcolor{gray!20}   $l=0$ & \cellcolor{gray!20} $l=1$ & \cellcolor{gray!20} $l=2$ &\cellcolor{gray!20}  $l=3$ & \cellcolor{gray!20} $l=4$   \\ 
\hline
min $\gamma$ & 0.9156  & 0.9058
  & 0.9023  & 0.9009 & 0.9003
\end{tabular}\label{tab:exampleDecayRateMD}
\end{table}


\section{Conclusions}\label{sec:Conclu}
We presented a graph-based construction of piecewise linear feedback controllers for discrete-time switched linear systems. The chosen graph structure influences the resulting LMI conditions, providing the user with flexible conditions in order to manage both numerical complexity and theoretical conservatism. The proposed techniques were demonstrated both in the robust case, where no knowledge of the switching signal is available, and in a mode-dependent case, where partial knowledge is utilized. The general graph framework also allowed to recover and generalize several results that already appeared in the literature. Future work will examine the conservatism introduced by different graph structures and broaden the scope of the method to handle a wider range of switched or hybrid systems. This will include exploring nonlinear dynamics and more complex switching situations.



\bibliographystyle{apacite}        
\bibliography{autosam}

\end{document}